\documentclass[reqno]{amsart}                     
\usepackage{amssymb,amsmath,mathrsfs}
\usepackage{epsfig,xcolor}
\usepackage{dsfont}
\usepackage{subfigure}
\usepackage{tikz}

\newtheorem{thm}{Theorem}
\newtheorem{lem}[thm]{Lemma}
\newtheorem{cor}[thm]{Corallary}

\newcommand\lemref[1]{{Lemma \ref{#1}}}
\newcommand\thmref[1]{{Theorem \ref{#1}}}

\newcommand\secref[1]{{Section \ref{#1}}}
\newcommand\figref[1]{{Fig.~\ref{#1}}}
\newcommand\tabref[1]{{Table \ref{#1}}}

\newcommand\paren[1]{{\left({#1}\right)}}
\newcommand\bigparen[1]{{\big({#1}\big)}}
\renewcommand\brace[1]{{\left\{{#1}\right\}}}
\renewcommand\brack[1]{{\left[{#1}\right]}}

\newcommand\inner[2]{{\left({{#1},{#2}}\right)_V}}

\newcommand\abs[1]{{\left|{#1}\right|}}
\newcommand\norm[1]{{\left\|{#1}\right\|}}

\newcommand\Set[1]{{\left\{{#1}\right\}}}
\newcommand\xarrow[1]{{\xrightarrow{#1}}}
\newcommand\toas{{\xarrow{\text{a.s.}}}}
\newcommand{\real}{}

\newcommand\indiam{{\rho}}
\newcommand\diam{{h}}
\newcommand\regconst{C}

\DeclareMathOperator*{\esssup}{ess\,sup}
\DeclareMathOperator*{\Diam}{diam}
\newcommand\f{\frac}
\newcommand\p{\partial}
\newcommand\y{\eta}
\newcommand\e{l}

\newcommand\ML{{\scriptsize\mbox{ML}}}

\newcommand\PML{{\scriptsize\mbox{P}}}
\newcommand\MC{{\scriptsize\mbox{MC}}}

\newcommand\EMC[1]{{E^\MC_{#1}}}
\newcommand\EML[1]{{E^\ML_{#1}}}
\newcommand\EPML[1]{{E^\PML_{#1}}}

\newcommand\vertex{{{a}}}
\newcommand\midpoint{{{m}}}
\newcommand\weight{{{w}}}

\newcommand\C{{{\mathcal C}}}
\newcommand\D{{e}}
\newcommand\B{{\mathcal B}}
\newcommand\T{{\mathcal T}}
\newcommand\cP{{\mathcal P}}
\newcommand\cE{{\mathcal E}}
\newcommand\cF{{\mathcal F}}
\newcommand\E{{\mathbb E}}
\newcommand\V{{\mathbb V}}
\newcommand\Var{{\V}}
\newcommand\Om{{\Omega}}
\newcommand\om{{\omega}}
\renewcommand\H{{\mathcal H}}
\renewcommand\P{{\mathbb P}}
\newcommand\R{{\mathbb R}}
\newcommand\N{{\mathbb N}}
\newcommand\bigO{{\mathcal O}}
\newcommand\eps{{\varepsilon}}
\newcommand\teps{{\tilde\varepsilon}}

\newcommand\hX{{\widehat X}}

\newcommand\dx{{\Delta x}}

\newcommand\Grad{\nabla}
\newcommand\Div{\nabla\cdot}

\newcommand\DistMesh{{\bf DistMesh\/}}
\newcommand\Matlab{{\bf Matlab\/}}
\begin{document}

\title[PMLMC for PDE with Random Input Data]{Projected Multilevel Monte Carlo Method for PDE with Random Input Data}

\author{Myoungnyoun Kim}
\address[Myoungnyoun Kim]{National Institute for Mathematical Sciences, Daejeon, Korea}
\email{gmnkim@nims.re.kr}
\author{Imbo Sim}
\address[Imbo Sim]{National Institute for Mathematical Sciences, Daejeon, Korea}
\email{imbosim@nims.re.kr}

\begin{abstract}
The order of convergence of the Monte Carlo method is $1/2$ which means that
we need quadruple samples to decrease the error in half in
the numerical simulation. 
Multilevel Monte Carlo methods reach the same
order of error by spending less computational time
than the Monte Carlo method.
To
reduce the computational complexity further,
we introduce
a projected multilevel Monte Carlo method.
Numerical experiments validate our theoretical
results.
\end{abstract}

\maketitle

\section{Introduction}
\label{sec:intro}

The partial differential equation (PDE) with random input data 
\cite{eeu:2007,fst:2005,gwz:2014,hl:2015,mk:2005}
takes a part of
the stochastic partial differential equation (SPDE)
which describes
the problem
with uncertain inputs \cite{gs:1991,mnst:2013,mnst:2014}
as follows
\begin{equation}\label{eq:SPDE}
\left\{
\begin{array}{l}
{\mathcal L}(x;\om) u(x;\om) = {\mathcal F}(x;\om),\quad x\in D,\quad \om \in \Om,
\\
\\
{\mathcal B}(x;\om) u(x;\om) = {\mathcal G}(x;\om), \quad x\in \p D,\quad \om \in \Om.
\end{array}
\right.
\end{equation}
Here, ${\mathcal L}$, ${\mathcal B}$, ${\mathcal F}$, and ${\mathcal G}$ 
are the stochastic partial differential operator, boundary operator, forcing term, 
and boundary value term in the spatial domain $D$ with its boundary $\p D$
and the range $\Om$ of input uncertainties, see 
\cite{BTZ:2004,BTZ:2005} 
for
the stochastic formulation of elliptic boundary value problems.
To ensure  the regularity of the solution of \eqref{eq:SPDE}, 
Babu\v{s}ka et al.\
\cite{BTZ:2004,BTZ:2005}
assume that 
the diffusion coefficients in $C^1(D)$
are bounded, uniformly coercive in the convex domain $D$
and their  first derivatives are bounded almost surely for
 stochastic elliptic PDEs.
We represent the uncertainties of the problem
by a random variable $\om$, which
follows some known or
unknown distribution of the probability. 
We can regard a SPDE as a problem depending on some parameters
which take certain values in finite ranges. In this parametrization context, 
\eqref{eq:SPDE}
 is a parametrized PDE (P$^2$DE) 
 \cite{asb:1978,balmes:1996,br:1995,ks:2011} 
 and
 one wants to simulate faster for a sequence of input data
 using the information on solutions at specially chosen parameters
 \cite{bmnp:2004,mpt:2002,ps:2010}.
  Both views have their own benefits to carry on
  and difficulties to cope with.
We adopt the view of a SPDE, which means that
we need to find out the mean of many solutions
corresponding to samples. 
Note that the total
cost of the computation is the number of samples  times the cost of
solving a PDE with fixed inputs after proper discretization,
see 
\cite{CGST:2011} 
for a detailed description on the 
cost to compute one sample.
  
  A simple 
  method to obtain the mean of  solutions is
  the Monte Carlo (MC) method, which requires 
  quadruple samples to reduce the current error in half.
   To reduce the cost while keeping the almost same 
  order of convergence, 
   variance reduction 
   techniques
    are suggested and studied, see 
    \cite{glasserman:2003}
    for many variants of them, like control variate, antithetic variate, 
   and so on. 
 Among them,
   the control variate introduces a variate whose mean is known. 
   When the correlation between
   the solution and the variate is good enough, 
   the optimal multiplier to reduce the variance
    is estimated by a few samples.
   In the two level MC method, 
   the selected variate is the difference between the solutions at  fine and coarse grids.
   The multiplier is just one. This can be extended to the case with
   more levels by observing that the telescoping
   sum of differences 
  between two consecutive solutions
   makes the finest solution.
   In this sense, the method gets its name, the Multilevel Monte Carlo (MLMC) method. 
  Note that we use the same sample to obtain the difference of approximations
   at  two consecutive grids, see 
   \cite{heinrich:2001} 
   for the idea 
   and the error estimation of
   the MLMC method
   applied to the parametric integration.
   There are many 
   results on MLMC methods applied to
   path simulations \cite{giles:2008},
   elliptic PDEs with random coefficients \cite{BSZ:2011,CST:2013,CGST:2011},
   stochastic elliptic multiscale PDEs \cite{abs:2013},
   parabolic SPDEs \cite{BLS:2013}.
   An extension of the MLMC method is derived by
   Giles et al.\ 
   \cite{GS:2014} 
   using 
  the antithetic variate method,
   which is successfully applied to
   eliminate L\'evy area simulation
   during the estimation of the payoff using the first order Milstein approximation.
    
      Stochastic collocation (SC) methods 
      \cite{BNT:2007,BNT:2010,ntw:2008:aniso,ntw:2008}
   are similar to MC methods except that
   their sample points are determined in the parameter space $\Om$, and
   an interpolant, for example, global Lagrange type polynomials as in 
  \cite{BNT:2007,BNT:2010}.
%
   In SC methods
   \cite{ntw:2008,xh:2005}, 
   they use the tensor product spaces for many random variables,
   which deteriorates the convergence rate and leads to the explosion of
   computation since the number of collocation points in a tensor grid
   grows exponentially. 
Smolyak 
\cite{smolyak:1963}
proposes sparse tensor product spaces  to reduce
the number of collocation points when the number of random variables
is moderately large.
The sparse tensor product grids are constructed by
either Clenshaw-Curtis 
\cite{cc:1960}
or Gaussian abscissas.
   Recently,
   Teckentrup et al.\
   \cite{TJWG:2014}
   introduce
    the 
Multilevel Stochastic Collocation (MLSC) approach for reducing the
cost of the SC method.
Inspired by multigrid solvers for linear equations,
the MLSC method uses a hierarchical sequence of spatial approximations combined with
stochastic discretization to minimize computational cost
under the conditions of finite dimensional noise and bounded random fields with
uniformly bounded and coercive coefficients,
see 
\cite{TJWG:2014}
for more details on the conditions.



In the MLMC application to SPDEs,
there are two main tasks  such that 
we take sample from the input random field, and
 form a spatial discretization of the PDE for a fixed sample parameter and solve it,
following the algorithm shown in 
\cite{CGST:2011}.
The former is the usual step for MC methods.
The latter causes an extra burden to implement the
MLMC algorithm, like storing elements and corresponding building blocks
for the stiffness matrix at the coarse grid.
In
this paper,
we propose a new 
MLMC estimator for 
PDE with random coefficients,
based on the original idea of the MLMC method
 in 
\cite{heinrich:2001},
which approximates the solution at the fixed sample in two
consecutive levels and takes the difference of them.
We solve the problem for a given sample at the fine grid, take
the projection of it to the corresponding coarse grid and regard
the projected solution as the approximation at the coarse grid.
Obviously, this procedure does not need to 
solve a problem
at the coarse grid which leads to the reduction of the 
computational cost.
 This replacement suggests us a name for this method, that is,
 the Projected Multilevel Monte Carlo (PMLMC) method 
 coming from the projection procedure of the method.
   In the PMLMC method,
   we project the solution at a fine
   grid into 
   the solution space at a coarse grid.
   The projection takes 
   less time than the solving a problems at the coarse grid.
   We provide the theorems on the order of convergence for the PMLMC method 
   and the optimal number of samples at corresponding levels.

In this paper, 
we consider the following model problem, for $d=1,2$,
\begin{equation}\label{eq:darcy}
-\Div \paren{k(x;\om) \Grad u(x)}=0,\quad x\in D=(0,1)^d,\quad \om \in \Om.
\end{equation}
In 1D,  we apply the Dirichlet boundary condition as the boundary condition
\begin{equation*}
 u(0)=1,\quad u(1)=0,
\end{equation*}
and introduce  additional Neumann boundary conditions in 2D, for $x=(x_1,x_2)$
\begin{equation*}
u(0,x_2)=1,\quad u(1,x_2)=0,\quad
\f{\p u}{\p n}(x_1,0)=\f{\p u}{\p n}(x_1,1)=0.
\end{equation*}
 Here,  an uncertain hydraulic coefficient of the Darcy flow
 is
based on a certain mean and covariance structure
inferred from the data
describing  the situation in the subsurface structure, 
see 
\cite{CGST:2011}
for details.
 There are several ways to represent the random variable $k$ in \eqref{eq:darcy} 
 by using
a Karhunen-Lo\`eve (KL)  \cite{gs:1991,loeve:1977,loeve:1978,stefanou:2009}
 a polynomial chaos \cite{hd:2014,xiu:2010,xk:2002:wiener,xk:2003}  
 and a wavelet expansion
 \cite{BSZ:2011}. 
 Since the coefficient in ground water flow
 can vary very largely, we can express them in a logarithmic scale 
 \cite{BNT:2007,BNT:2010,gd:1998,HK:1985}. 
 In this case, we can apply any expansion for
 the logarithm of the coefficient, instead of the coefficient itself,
 which leads to an exponential dependence of the coefficient on the
 random variable $\om$ and its second moment might be unbounded
 \cite{BNT:2007,BNT:2010}.
%
 For simplicity,
we expand
the logarithm of 
the random conductivity 
in 1D 
through the KL expansion
as
Cliffe et al.\ 
\cite{CGST:2011},
\begin{equation*}
\log k(x;\omega)= 
\E\paren{\log k(x;\cdot)}
+
\sum_{n=1}^{\infty}\sqrt{\theta_n}\xi_n(\om) \phi_n(x).
\end{equation*}
%
Here, $\Set{\xi_n}_{n=1}^\infty$ is a set of
zero mean random variables uncorrelated to each other.
The eigenvalues $\Set{\theta_n}_{n=1}^\infty$ and
normalized eigenvectors $\Set{\phi_n}_{n=1}^\infty$
are generated from the covariance operator defined by
\begin{equation}
\label{eq:cov}
C(x,y)=\sigma^2 \exp\paren{-\f{\norm{x-y}_p}{\lambda}},\quad
x,y\in 
D,
\end{equation}
where $\sigma^2$ is the variance, $\lambda$ is the
correlation length and $\norm{\cdot}_p$ is the usual $p$-norm. 
By the choice of $p=1$, $\sigma^2\ge 1$ and $\lambda\le\Diam D$,
 the coefficient $k$
 is homogeneous and from Kolmogorov's theorem
  in 
  \cite{prato:1992},
  it belongs to $C^{0,\eta}(D)$ almost surely with
  $\eta<1/2$.

     A theoretical analysis of elliptic PDEs with random coefficients 
   such as \eqref{eq:darcy}  is
  done in 
  \cite{BSZ:2011} 
  under the condition that coefficient
  fields in $W^{1,\infty}(D)$ are bounded uniformly from above and
  away from zero. Charrier et al.\ 
  \cite{CST:2013} 
  analyze when
  the coefficient is not uniformly bounded and only in $C^{0,\eta}(D)$
  with $\eta<1/2$.
  We follow the covariance relation in
  \cite{CGST:2011}
   and
  show the results on the estimation of $\E\brack u$ of the solutions 
  of \eqref{eq:darcy}
  with respect to the realizations of the randomness.
  We use the same condition in Cliffe et al.\ 
  \cite{CGST:2011}
   to prove the order of convergence
  and the optimal number of samples  for the PMLMC method.
   
This paper has the following structure.
In \secref{sec:pre}, we provide all the preliminaries on 
 probability and Bochner integrals.
We analyze the order of convergence and optimal number of samples for
 MC, MLMC, and PMLMC methods in \secref{sec:order}.
The analysis on the variance of the projection is also provided in \secref{sec:order}
with a corollary
for a hierarchical grid structure.
In \secref{sec:NS},
numerical results of \eqref{eq:darcy}  are 
illustrated
 on the order of convergence and cost savings. 

\section{Preliminaries}
\label{sec:pre}

Let $\paren{\Om,\H,\P}$ be a {\em probability space\/} where
$\Om$ is the { sample space\/} and its elements $\om\in\Om$ are  
{ outcomes\/}, 
$\H\subset 2^\Om$ is the 
{ grand history\/} of $\sigma$-algebra and its elements $H\in\H$ are
 { events\/}, and $\P:\H\to[0,1]$ is the 
 { probability measure\/}.
 For a measurable space $\paren{E,\cE}$ with the $\sigma$-algebra $\cE\subset 2^E$,
 a mapping $X:\Om\to E$ satisfying $X^{-1} A\in\H$ for any $A\in\cE$,
  is an $E$-valued random variable.
The image $\mu_X$ of $\P$ under an $E$-valued 
random variable $X$,
\begin{equation*}
\mu_X\paren A = \P\brack{X^{-1}\paren A} = \P\brack{X\in A},\quad {}^\forall A\in\cE,
\end{equation*}
is a probability measure on 
$\paren{E,\cE}$,
called the {\em distribution\/} of $X$.
If $\paren{F,\cF}$ is  measurable,
$E$-valued and $F$-valued
random variables $X$ and $Y$ are {\em independent\/} if
their { joint distribution\/} $\mu_{X,Y}$ is the product measure
$\mu_X\times\mu_Y$
on $\paren{E\times F,\cE\otimes\cF}$,
 where $\mu_X$ 
and 
$\mu_Y$ 
are the marginal distributions of
$X$ and $Y$ respectively.
A {\em simple\/} $E$-valued  random variable $X$ 
attains
only a finite number of distinct values
$\Set{x_n}_{n=1}^N\subset E$, and has the form
$X=\sum_{n=1}^N x_n\,1_{A_n}$ where the $A_n\in\H$ are disjoint and
$1_{A_n}$ is the indicator function of $A_n$, 
see 
\cite{cinlar:2011}
 for more details. 

Let $V$ be a separable Banach space with the norm $\norm{\cdot}_V$,
its topological dual $V'$ and
the Borel $\sigma$-algebra $\B\paren V$ for the measurable space
$\paren{V,\B\paren V}$.
If there exists a sequence $\Set{X_n}$
of simple $V$-valued  random variables such that $X_n\toas X$,  that is,
\begin{equation*}
\P\brack{\lim_{n\to\infty}\norm{X_n-X}_V=0}=1,
\end{equation*}
then  $X$ is said to be {\em strongly\/} measurable.
From 
\cite{prato:1992}, 
a $V$-valued  random variable has a sequence $\Set{X_n}$
of simple $V$-valued  random variables such that,
for arbitrary $\om\in\Om$, the sequence
$\Set{\norm{X\paren\om-X_n\paren\om}_V}$ is monotonically
decreasing to $0$. This means that
$\Set{\lim_{n\to\infty}\norm{X-X_n}_V=0}=\Om$ 
or
$\P\brack{\lim_{n\to\infty}\norm{X_n-X}_V=0}=1$,
i.e., $X_n\toas X$.
Thus
any
$V$-valued  random variable is strongly measurable if $V$ is separable.

For a simple $V$-valued  random variable $X=\sum_{n=1}^N x_n\,1_{A_n}$,
if $\P\brack{A_n}$ is finite whenever $x_n\ne0$, then
$X$ is integrable and its integral, called the 
{\em Bochner integral\/} $\E\brack{X}$ of $X$,
is
\begin{equation*}
\E\brack{X}
=
\int_\Om X\,d\P
=
\int_\Om X\paren\om\,\P\paren{d\om}
=
\sum_{n=1}^N x_n\,\P\brack{A_n}.
\end{equation*}

As in  
\cite{prato:1992},
 the real valued random variable $\norm{X\paren\cdot}_V$ is measurable for any
$V$-valued random variable $X$. Then
a $V$-valued random variable $X$ is Bochner integrable if
\begin{equation*}
\E\brack{\norm{X}_V}
=
\int_\Om \norm{X}_V\,d\P
=
\int_\Om\norm{X\paren\om}_V\,\P\paren{d\om} < \infty.
\end{equation*}
Since a $V$-valued random variable $X$ is strongly measurable, 
 there is
a sequence $\Set{X_n}$
of simple $V$-valued  random variables  such that
$X_n\toas X$, then 
\begin{equation*}
Y_n\paren\om
=
\left\{
\begin{array}{ll}
X_n\paren\om,\quad
&\mbox{if }\norm{X_n\paren\om}_V\le \f32 \norm{X\paren\om}_V,
\\
\\
0,&\mbox{otherwise},
\end{array}
\right.
\end{equation*}
forms a sequence $\Set{Y_n}$ of simple $V$-valued random variables 
such that $\norm{Y_n}_V\le \f32 \norm{X}_V$ and $Y_n\toas X$.
Since $\norm{X}_V$ is integrable,
by Dominated Convergence Theorem, 
\begin{equation*}
\lim_{n\to\infty} \E\brack{\norm{X-Y_n}_V}
=
\E\brack{\lim_{n\to\infty}\norm{X-Y_n}_V }
=0,
\end{equation*}
holds, and $\Set{\E\brack{Y_n}}$ is a Cauchy sequence in $V$.
Then the Bochner integral 
of $X$ is
\begin{equation*}
\E\brack X
=
\int_\Om X\,d\P
=
\lim_{n\to\infty}
\int_\Om Y_n\,d\P
=
\lim_{n\to\infty}\E\brack {Y_n},
\end{equation*}
and the limit is the same
for any sequence $\Set{X_n}$
of simple $V$-valued  random variables  satisfying 
$\lim_{n\to\infty} \E \brack{\norm{X - X_n}_V} =0$.

Let $L^p\paren{\Om;V}$ be the {\em Bochner space\/} of Bochner integrable, 
$V$-valued random variables $X$
such that the corresponding norm
\begin{equation*}
\norm{X}_{L^p\paren{\Om;V}}
=
\left\{
\begin{array}{ll}
\paren{\,\E\brack{\norm{X}^p_V}\,}^{\f1p},
&
\quad
1\le p <\infty,
\\
\\
\esssup_{\om\in\Om} \norm{X\paren\om}_V,
&\quad  p=\infty,
\end{array}
\right.
\end{equation*}
is finite, where $X$ is the equivalence class  with respect to the equivalence relation
$X\sim Y$ if and only if $X=Y$ almost surely.

If $V$ is a separable Hilbert space with the inner product 
$\inner\cdot\cdot$,
then the Bochner integral 
of the inner product 
of two independent 
$V$-valued
random
variables $X$ and $Y$ equals 
the inner product of their Bochner integrals,
\begin{eqnarray}\label{eq:Bochner:inner}
\E\brack{\inner XY}
&=&
\int_{V\times V}  
\inner{x}{y}\,\mu_{X,Y}\paren{dx,dy}
=
\int_{V\times V}  
\inner{x}{y}\,\mu_{X}\paren{dx}\,\mu_{Y}\paren{dy}
\nonumber
\\
&=&
\int_{\Om\times\Om} 
\inner{X\paren{\om_1}}{Y\paren{\om_2}}\,\P\paren{d\om_1}\,\P\paren{d\om_2}
\nonumber
\\
&=&
\int_{\Om}
\brack{
 \int_{\Om}
\inner{X\paren{\om_1}}{Y\paren{\om_2}}\,\P\paren{d\om_2}}
\P\paren{d\om_1}
\nonumber
\\
&=&
\int_{\Om}
\brack{
\inner{X\paren{\om_1}}{\int_{\Om} Y\paren{\om_2}\,\P\paren{d\om_2}}}
\P\paren{d\om_1}
\nonumber
\\
&=&
\inner{\int_{\Om} X\paren{\om_1}\, \P\paren{d\om_1}}
        {\int_{\Om} Y\paren{\om_2}\,\P\paren{d\om_2}}
\nonumber
\\
&=&
\inner{\E\brack X}{\E\brack Y}.
\end{eqnarray}
Here the classical Fubini theorem and 
Proposition 1.6 in Chapter 1 of
\cite{prato:1992}
 are used
as well as the property of the independence. 
Finally, the {\em variance\/} of a $V$-valued random variable is
\begin{eqnarray*}
\Var\brack{X}
&=&
\E\brack{\norm{X-\E\brack X}_V^2}
=
\E\brack{\norm{X}_V^2 - 2\real\inner{X}{\E\brack X}+\norm{\E\brack X}_V^2}
\nonumber
\\
&=&
\E\brack{\norm{X}_V^2}-2\norm{\E\brack X}_V^2+\E\brack{\norm{\E\brack X}_V^2}
\nonumber
\\
&=&
\E\brack{\norm{X}_V^2}-\norm{\E\brack X}_V^2,
\end{eqnarray*}
which coincides with the usual definition of the variance.


\section{Order of Convergence and Complexity}
\label{sec:order}

When $\lim_{N\to\infty}\E\brack{X_N}=\E\brack X$ 
for some random variable $X$ and its 
approximation $X_N$, 
the 
{\em order of convergence\/}
 is $p$ if
$\norm{\E\brack{X_N-X}}\le C N^{-p}$, where
$C$ is independent of $N$ and $\norm{\cdot}$ is the 
proper norm under the given context of convergence,
see 
\cite{CGST:2011}.

 The computational cost $\C(\hX_N)$ is
 the number of floating point operations to compute
 $\hX_N$ which is the realization of $X_N$.
Implicitly, 
$\hX_N$ should satisfy some criteria of 
convergence, for example, the error between
$\hX_N$ and $X_N$
is less than or equal to 
the given number $\eps$. This means that the trivial
choice of $\hX_N$ should be excluded.

Let
$V=H^1_0(D)$ be a separable Hilbert space 
with the inner product $\inner{\cdot}{\cdot}$ and its associated
norm $\norm{\cdot}_V$ such that
\begin{equation*}
\inner u v = \int_D\paren{\Grad u\cdot \Grad v + u v}\,dx,\quad
\norm{u}_V=\sqrt{\inner uu}.
\end{equation*}

\subsection{Monte Carlo Method}
\label{subsec:MC}

Let $\Set{u_k}_{k=1}^N \subset V$ be solution samples corresponding
to $N$ independent, identically distributed realizations of 
random input data,
and $\EMC{N}(u)$  the mean
of them by the 
{\em Monte Carlo (MC)\/} method defined as
\begin{equation*}
\EMC{N}(u)= \f1N\sum_{k=1}^N u_k \in V.
\end{equation*}
The MC estimator satisfies the {\em unbiased\/} property as follows
\begin{equation*}
\E\brack{\EMC{N}(u)}
=\f1N\sum_{k=1}^N \E\brack{u_k}
=\f1N\sum_{k=1}^N \E\brack{u}
=\E\brack u,
\end{equation*}
since $\Set{u_k}_{k=1}^N $ are  independently chosen
following the identical distribution.
And the variance of the MC estimator is
\begin{equation*}
\Var\brack{\EMC{N}(u)}
=\E\brack{\norm{\EMC{N}(u) - \E\brack{\EMC{N}(u)} }_V^2}
=\E\brack{\norm{\EMC{N}(u) - \E\brack u }_V^2}.
\end{equation*}
The error $\D^\MC_N(u)$ between 
$\EMC{N}(u)$ and $\E \brack u\in V$,
\begin{equation*}
\D^\MC_N(u)=
 \EMC{N}(u) - \E\brack u
=
 \f1N\sum_{k=1}^N u_k  - \E\brack u
=\f1N\sum_{k=1}^N \paren{u_k - \E\brack u}
\end{equation*}
has independent terms $u_k - \E\brack u$.  
Using \eqref{eq:Bochner:inner}, we know that
the Bochner integral
of the inner product between mutually independent terms
$u_k - \E\brack u$ and $u_j - \E\brack u$  for $k\ne j$, becomes
\begin{equation*}
\E \brack{\inner {u_k - \E\brack u}{u_j - \E\brack u}}
=\inner {\E\brack{u_k - \E\brack u}}{\E\brack{u_j - \E\brack u}}
=0, 
\end{equation*}
since $\E\brack{u_k - \E\brack u}=\E \brack {u_k} - \E\brack u =0$ 
from the 
unbiased property of  
the MC method. Then
the mean square error is
\begin{eqnarray}
\label{eq:MC}
\norm{\D^\MC_N (u)}^2
&=&
\E\brack{\norm{\EMC{N}(u) - \E\brack u}_V^2}
=\Var\brack{\EMC{N} (u) }
\nonumber
\\
&=&
{
\f1{N^2}
\E\brack{\norm{\sum_{k=1}^N 
\paren{ u_k - \E\brack u }}^2_V } }
=
{
\f1{N^2}
\sum_{k=1}^N  
\E\brack{\norm{u_k - \E\brack u }^2_V }
}
\nonumber
\\
&=&
{\f1{{N}}{\E\brack{\norm{u - \E\brack u }^2_V} } }
=
{\f1{{N}}{\V\brack u } }
\le
{\f1{N}{\E\brack{\norm{u}^2_V}}}
=
\f1{{N}} {\norm{u}}^2,
\end{eqnarray}
where $\norm{\cdot}=\norm{\cdot}_{L^2\paren{\Om;V}}$ is the norm of $L^2\paren{\Om;V}$.
Thus,
the relative error 
of the MC method is
less than or equal to 
$1/\sqrt{N}$,
that is, the order of convergence  of  the MC method is $1/2$.
Precisely,
let $\eps$ be the desired error bound for the MC estimator,
 i.e., 
$\|{\D_N^\MC\paren u}\| \le \norm{u}/{\sqrt{N}} = \eps$,
 then 
$N={{\norm{u}}^2\eps^{-2}}$ is the best choice
to attain the desired error.
That is, we must increase
the number of samples  fourfold to
decrease the error by half.

\subsection{Single Level Monte Carlo Method}
\label{subsec:SLMC}

Let $\T_l$ be the 
{\em triangulation\/} of $D$ into simplices with a mesh size $h_l$
for 
$l\in\N$,
 and
nodes in $\T_{l-1}$ belong to those in $\T_l$ which ensures the 
{\em hierarchical\/}
structure of triangulation.
Let $V_l$ be the space of piecewise linear functions on $D$,
 i.e.,
\begin{equation*}
V_l = \Set{v\in V\,:\, v\big|_K\in \cP_1(K),\, {}^\forall K\in\T_l},
\end{equation*}
where $\cP_1(K)$ is a linear polynomial space on a triangle $K\in\T_l$.
Let $\{u_{l,k}\}_{k=1}^{N_l}\subset V_l$ be Galerkin
Finite Element approximations in $V_l$ 
corresponding to the realizations of the random coefficient.
Then 
the 
{\em Single Level
Monte Carlo\/} Finite Element approximation $\EMC{l,N_l}(u)$ in $V_l$
is defined by
\begin{equation*}
\EMC{l,N_l}(u)
=
\EMC{N_l}(u_l)
=\f1{N_l}\sum_{k=1}^{N_l}u_{l,k}
\in L^2\paren{\Om;V_l}.
\end{equation*}
From the equality relation 
of \eqref{eq:MC}, 
the variance of $\EMC{l,N_l}(u)$ is
\begin{eqnarray*}
\Var\brack{\EMC{l,N_l}(u)}
=
\Var\brack{\EMC{N_l}(u_l)}
=
\f1{N_l}\Var\brack{u_l}.
\end{eqnarray*}
Let $\D^\MC_{l,N_l}(u)$ be the error between $\EMC{l,N_l}(u)$ and $\E\brack u$,
\begin{equation*}
\D^\MC_{l,N_l}(u)=\EMC{l,N_l}(u)-\E\brack u
=\EMC{l,N_l}(u)-\E\brack{\EMC{l,N_l}(u)}+\E\brack{u_l}-\E\brack u.
\end{equation*}
Then we expand the mean square error $\norm{\D^\MC_{l,N_l}(u)}^2$ as follows
\begin{eqnarray*}
\norm{\D^\MC_{l,N_l}(u)}^2
&=&
{
\E\brack{\norm{\EMC{l,N_l}(u) - \E\brack{ \EMC{l,N_l}(u)}
 +\E\brack{ u_l} - \E\brack u }^2_V }
}
\nonumber
\\
&=&
{
\E\brack{\norm{\EMC{l,N_l}(u) - \E\brack{ \EMC{l,N_l}(u)} }^2_V}
+
\E\brack{\norm{\E\brack{ u_l}  -\E\brack u}^2_V}
}
\nonumber
\\
&+&
2\real\bigparen
{
\E\brack{\inner{ \EMC{l,N_l}(u) - \E\brack{ \EMC{l,N_l}(u)} }{\E\brack{ u_l}  -\E\brack u}}
}
\nonumber
\\
&=&
\Var\brack{\EMC{l,N_l}(u)}
+
{\norm{\E\brack{ u_l}- \E\brack{u}}^2}
=
\f1{N_l}\Var\brack{u_l}
+
{\norm{\E\brack{ u_l- u}}^2}.
\end{eqnarray*}
Here, the Bochner integrals of inner products between different deviations are zero
due to the 
unbiased property
of the MC method
and the relation \eqref{eq:Bochner:inner} for
two independent random variables. 
The computational cost 
of the estimator by the Single Level Monte Carlo method is
\begin{equation*}
\C\paren{\EMC{l,N_l}(u)} = N_l\C_l,
\end{equation*}
where $\C_l$ is the mean
computational complexity at level $l$ by the finite element method.
The main difference to the MC method is the use of the finite space to
approximate a solution for a realization of the randomness, 
which 
results
the approximation error.

\subsection{Multilevel Monte Carlo Method}
\label{subsec:MLMC}

Set $u_0=0$, and $w_l =u_l-u_{l-1}$ for $l=1,\ldots,L$, where
$L$ is the maximum level. Clearly,
$u_L=\sum_{l=1}^L\paren{u_l-u_{l-1}}$ and further we have
\begin{equation*}
\E\brack{ u_L }= \sum_{l=1}^L\E\brack{u_l-u_{l-1}}
=\sum_{l=1}^L\E\brack{ w_l}
=\E\brack{u_1}+\sum_{l=2}^L\E\brack{ w_l}.
\end{equation*}
From the above observation,
the 
{\em Multilevel Monte Carlo (MLMC)\/}
 Finite Element approximation 
$\EML{L}(u)$
is defined by
\begin{equation*}
\EML{L}(u)
=
\sum_{l=1}^L \EMC{N_l}\paren{w_l}
=
\EMC{N_1}\paren{u_1}+
\sum_{l=2}^L \EMC{N_l}\paren{w_l}.
\end{equation*}
Since $\E\brack{ \EMC{N_l}(w_l)}=\E\brack{ w_l}$ 
by the 
unbiased property 
of the MC method,
we have
\begin{equation*}
\E\brack{ \EML{L}(u)}
=\sum_{l=1}^L \E\brack{\EMC{N_l}\paren{w_l}} =\sum_{l=1}^L \E \brack{w_l}
=\E\brack{ u_L}.
\end{equation*}
The variance of the MLMC estimator is
\begin{eqnarray*}
\Var\brack{\EML{L}(u)}
&=&
\E\brack{\norm{\EML{L}(u)-\E\brack{\EML{L}(u)} }_V^2 }
\nonumber
\\
&=&
\E\brack{\norm{\sum_{l=1}^L \paren{\EMC{N_l}\paren{w_l}-\E \brack{w_l}} }_V^2 }
\nonumber
\\
&=&
\sum_{l=1}^L  \E\brack{\norm{\EMC{N_l}\paren{w_l}-\E \brack{w_l} }_V^2 }
\nonumber
\\
&+&
2\real\paren{\sum_{l\ne k}
\inner{\EMC{N_l}\paren{w_l}-\E \brack{w_l} }{\EMC{N_k}\paren{w_k}-\E \brack{w_k} } }
\nonumber
\\
&=&
\sum_{l=1}^L  \Var\brack{\EMC{N_l}\paren{w_l}}
=
\sum_{l=1}^L  \f1{N_l}\Var\brack{w_l}
=
\sum_{l=1}^L  \f1{N_l}\Var\brack{u_l-u_{l-1}},
\end{eqnarray*}
since the deviations $\EMC{N_l}\paren{w_l}-\E \brack{w_l}$ 
are mutually independent due to different
samples.
Let $\D^\ML_L(u)$ be the error between 
$\EML{L}(u)$ and $\E\brack u$
\begin{equation*}
\D^\ML_L(u)=\EML{L}(u)-\E\brack u
=\EML{L}(u)-\E\brack{\EML{L}(u) }+\E\brack{u_L}-\E\brack u.
\end{equation*}
The mean square error becomes
\begin{eqnarray*}
\norm{\D^\ML_L(u)}^2
&=&
\E\brack{\norm{\EML{L}(u)  - \E\brack{\EML{L}(u) } + \E\brack{  u_L} - \E\brack u }_V^2}
\nonumber
\\
&=&
\E\brack{\norm{\EML{L}(u)  - \E\brack{\EML{L}(u) }  }_V^2 }
+
\E\brack{\norm{\E\brack{  u_L} - \E\brack u  }_V^2 }
\nonumber
\\
&+&
2\real \E\brack{\inner{\EML{L}(u)  - \E\brack{\EML{L}(u) }}{\E\brack{  u_L} - \E\brack u } }
\nonumber
\\
&=&
\Var\brack{\EML{L}(u)} + \norm{\E\brack{  u_L} - \E\brack u  }^2
\nonumber
\\
&=&
\sum_{l=1}^L  \f1{N_l}\Var\brack{u_l-u_{l-1}}
+
\norm{\E\brack{  u_L- u}  }^2.
\end{eqnarray*}
The computational cost $\C({\EML{L}(u)})$
of the MLMC estimator  is
\begin{equation*}
\C\paren{\EML{L}(u)} = \sum_{l=1}^L N_l \paren{\C_l+\C_{l-1}}
=N_L \C_L + \sum_{l=1}^{L-1} \paren{N_l+N_{l+1}} \C_l ,
\quad\C_0=0,
\end{equation*}
 where $\C_l$ contains the mean complexity at level $l$
 including differencing cost. 
Cliffe et al.\
 \cite{CGST:2011}
 discuss the  cost saving
 in the decay tendency of the variances,
 compared to the MC method.
 
\subsection{Projected Multilevel Monte Carlo Method}
\label{subsec:PMLMC}

Let $P_l$ be the 
{\em projection\/} from $V_l$ into $V_{l-1}$, i.e.,
\begin{equation}
\label{eq:proj}
\inner{P_l u_l}{v}=\inner{u_l}{v},\quad {}^\forall v\in V_{l-1},
\end{equation}
 and $P_1=0$. 
Set $w_l=\paren{I-P_l}u_l$, then
the 
{\em Projected Multilevel Monte Carlo (PMLMC)\/} 
Finite Element approximation $\EPML{L}(u)$ is 
\begin{equation*}
\EPML{L}(u) = \sum_{l=1}^L {\EMC{N_l}(\paren{I-P_l} u_l) }
=
\sum_{l=1}^L {\EMC{N_l}(w_l) }
=
\EMC{N_1}(u_1) 
+
\sum_{l=2}^L {\EMC{N_l}(w_l) }.
\end{equation*}
Note that we replace $\EMC{N_{l}}(u_{l-1})$ in the MLMC method
by the projection $P_l  \EMC{N_l}(u_l)$ of $\EMC{N_l}(u_l)$ 
in order to
simplify the computational complexity. 
Using the unbiased 
property of the MC method, we have
\begin{equation*}
\E\brack{\EPML{L}(u)}
=
\sum_{l=1}^L {\E\brack{\EMC{N_l}(w_l) } }
=
\sum_{l=1}^L  \E\brack{w_l }
=
\E\brack{u_1 }
+
\sum_{l=2}^L  \E\brack{w_l }.
\end{equation*}
Since $\E[{\EMC{N_l}(w_l)}]=\E\brack{w_l}$,
the variance of the PMLMC estimator is
\begin{eqnarray*}
\Var\brack{\EPML{L}(u)}
&=&
\E\brack{\norm{\EPML{L}(u)-\E\brack{\EPML{L}(u)} }_V^2 }
\nonumber
\\
&=&
\E\brack{\norm{\sum_{l=1}^L  \paren{\EMC{N_l}(w_l) - \E\brack{w_l } }  }_V^2 }
\nonumber
\\
&=&
\sum_{l=1}^L 
\E\brack{
\norm{ \paren{\EMC{N_l}(w_l) - \E\brack{w_l } } }_V^2
}
\nonumber
\\
&+&
2\real
\paren{
\sum_{l\ne k}
\E\brack{
\inner{\EMC{N_l}(w_l) - \E\brack{w_l } }
{\EMC{N_k}(w_k) - \E\brack{w_k } }
}
}
\nonumber
\\
&=&
\sum_{l=1}^L 
\Var\brack{\EMC{N_l}(w_l)}
=
\sum_{l=1}^L 
\f1{N_l}
\Var\brack{w_l}
=
\sum_{l=1}^L 
\f1{N_l}
\Var\brack{u_l-P_l u_l},
\end{eqnarray*}
since 
we use different samples at each level, 
in other words, the deviations are independent to each other.
The error of the PMLMC estimator is defined by
\begin{equation*}
\D^\PML_L(u) = \EPML{L}(u) - \E\brack u.
\end{equation*}
Then the mean square error is
\begin{eqnarray*}
\norm{\D^\PML_L(u)}^2
&=&
\E\brack{\norm{\EPML{L}(u)  - \E\brack{\EPML{L}(u) } + \E\brack{\EPML{L}(u) } - \E\brack u }_V^2}
\nonumber
\\
&=&
\E\brack{\norm{\EPML{L}(u)  - \E\brack{\EPML{L}(u) } }_V^2}
+ 
\E\brack{\norm{\E\brack{\EPML{L}(u) } - \E\brack u }_V^2}
\nonumber
\\
&+&
2\real
\bigparen{
\E\brack{
\inner{\EPML{L}(u)  - \E\brack{\EPML{L}(u) } }{\E\brack{\EPML{L}(u) } - \E\brack u }
}
}
\nonumber
\\
&=&
\Var\brack{\EPML{L}(u)}
+ 
{\norm{\E\brack{\EPML{L}(u) } - \E\brack u }^2}
\nonumber
\\
&=&
\sum_{l=1}^L \f1{N_l}\Var\brack{u_l-P_l u_l}
+ 
{\norm{\E\brack{\EPML{L}(u) - u} }^2}.
\end{eqnarray*}
Using the regular triangulation condition
briefly stated in \secref{subsec:var},
see 
\cite{BA:1976,ciarlet:1978,jamet:1976,zlamal:1968}
for various conditions,
we can bound the first variance term
as follows.
\begin{lem}\label{lem:var:proj}
Let $\T_{l}$ be a triangulation of $D$ to form an approximate space $V_{l}$,
and $\{\vertex_i\}_{i=1}^3$ the set of vertices of a triangle $K\in\T_{l-1}$.
For a fine grid solution $u_l\in V_l$,
the bound of the variance of $u_{l}-P_l u_l$ is
\begin{equation*}
\Var\brack{u_{l}-P_l u_l} \le 
\regconst
\paren{ \diam_{l-1}^2 +1 } \sum_{K\in\T_{l-1}}\sum_{i=1}^3
\Var\brack{\paren{u_l-P_l u_l}(\midpoint_i(K);\cdot)},
\end{equation*}
where $C$ is a constant related to the regular triangulation,
$\diam_{l-1}$ is the maximum of diameters of triangles,
$\Set{\midpoint_i(K)}_{i=1}^3$ are  mid points on edges of 
the triangle,
and variances $\Var\brack{\paren{u_l-P_l u_l}(\midpoint_i(K);\cdot)}$ 
are  usual variances for discrete values..
\end{lem}
\begin{proof}
We can expand
the variance of $u_{l}-P_l u_l$ as follows 
\begin{eqnarray*}
\Var\brack{u_{l}-P_l u_l}
&=&
\E\brack{\norm{u_{l}-P_l u_l-\E\brack{u_{l}-P_l u_l}}_V^2 }
\\
&=&
\E\brack{\norm{v}_{L^2(D)}^2 }
+
\E\brack{\norm{\Grad\,v}_{L^2(D)^2}^2 },
\end{eqnarray*}
by introducing an auxiliary deviation $v=u_{l}-P_l u_l-\E\brack{u_{l}-P_l u_l}$.
We obtain the bound
in \secref{subsec:var} using the regular triangulation condition.
\end{proof}
Now, we want to bound the error 
with respect to the approximation error.
\begin{thm}
\label{thm:PML}
The error $\D^\PML_L(u)$ 
is bounded by
\begin{equation} \label{eq:PML}
\norm{\D^\PML_L(u)}
\le
c_0\norm{u}+\sum_{l=1}^{L} c_l \norm{u-u_l},
\end{equation}
where the constants $c_l$ are dependent on numbers of samples as 
\begin{equation*}
\left\{
\begin{array}{l}
c_0={{N_1}^{-\f12}+2\sum_{l=2}^L N_l^{-\f12} },
\\
c_1= 1+N_1^{-\f12}+2 N_2^{-\f12},
\\
c_l=2+ 2{{N_l}^{-\f12}}+2{{N_{l+1}}^{-\f12}}, \mbox{ for } l=2,\ldots,L-1,
\\
c_L=2+2{{N_L}^{-\f12}}.
\end{array}
\right.
\end{equation*}
\end{thm}
\begin{proof}
The difference between $P_l  \EMC{N_l}(u_l)$ and $\EMC{N_{l}}(u_{l-1})$ is
\begin{eqnarray*}
&&P_l  \EMC{N_l}(u_l) - \EMC{N_{l}}(u_{l-1})
\nonumber
\\
&=&
P_l  \EMC{N_l}(u_l) - P_l \E\brack{ u_l }
  +P_l \E\brack{ u_l} - \E\brack{ u_{l-1}} +\E\brack{ u_{l-1}} -\EMC{N_{l}}(u_{l-1})
  \nonumber
  \\
  &=&-P_l  \D^\MC_{N_l}(u_l)  +  P_l \E\brack{u_l -u_{l-1} }+\D^\MC_{N_{l}}(u_{l-1})
\end{eqnarray*}
for $l=2,\ldots,L$
since $P_l=I$ on $V_{l-1}$. 
Then the norm $\|{\EML{L}(u) - \EPML{L}(u)}\|$ can be bounded by
\begin{eqnarray*}
\norm{\EML{L}(u) - \EPML{L}(u)}
&=&
\norm{\sum_{l=2}^L \paren{P_l  \EMC{N_l}(u_l) - \EMC{N_{l}}(u_{l-1}) }}
 \nonumber
\\
&\le&
 \sum_{l=2}^L 
 \paren{
 \norm{ \D^\MC_{N_l}(u_l)  }
+
 \norm{ \E\brack{u_l-u_{l-1}}  }
+
 \norm{  \D^\MC_{N_l}(u_{l-1})  } }
 \nonumber
 \\
&\le&
 \sum_{l=2}^L 
 \paren{
 \norm{ \D^\MC_{N_l}(u_l)  }
+
 \norm{ {u_l-u_{l-1}}  }+
 \norm{  \D^\MC_{N_l}(u_{l-1})  } },
\end{eqnarray*}
and 
the error $\D^\PML_L(u)$ 
can be bounded by
\begin{eqnarray*}
\norm{\D^\PML_L(u)}
&\le&
{
\norm{\E\brack u -\EML{L}(u)}
 + \norm{\EML{L}(u)-\EPML{L}(u)} }
\nonumber
 \\
&\le&
\norm{ {u -   u_L} }+
\sum_{l=1}^L \frac1{\sqrt{N_l}}
{\norm{u_l -u_{l-1}}}
   \nonumber
\\
&+&
\sum_{l=2}^L 
 \brace{\frac1{\sqrt{N_l}} \paren{\norm{u_l} + \norm{u_{l-1}} }+%
 \norm{ {u_l-u_{l-1}}  } }
 \nonumber
 \\
&\le&
c_0\norm{u}+\sum_{l=1}^{L} c_l \norm{u-u_l},
\end{eqnarray*}
in terms of $\norm{u}$, $\norm{u-u_l}$ and $\norm{u-u_{l-1}}$
from  expansions of $\norm{u_l}=\norm{u_l-u+u}$ and
$\norm{ {u_l-u_{l-1}}  } = \norm{ {u_l-u+u-u_{l-1}}  }$.
\end{proof}
%
When $\norm{u}$ is a little bit large,
its effect can be negated by increasing
the total number of samples as shown in $c_0$ in \thmref{thm:PML},
since it depends only on the number of samples at each level.
On the other hand, we can decrease the error only when the approximation
error at each level has a good order of convergence from the constant
dependence on $\norm{u-u_l}$.
Furthermore,
if $u_l$ converges to $u$ in mean square, then $\Var\brack{w_l}=\Var\brack{\paren{I-P_l}u_l}$
converges to zero as the level $l$ increases, and fewer samples are needed on finer levels
to estimate $\E\brack{w_l}$.
Compared to the error bound of 
the MLMC method, 
that of 
the PMLMC method has more dependence on 
numbers of samples $N_l$ for $l>1$ to
control the error $\norm{u}$.
The approximate property at the level next to the finest grid should be better to ensure
good rate of convergence.

The computational cost $\C({\EPML{L}(u)})$ of the PMLMC estimator $\EPML{L}(u)$
 is 
\begin{equation*}
\C\paren{\EPML{L}(u)} = \sum_{l=1}^L N_l \C_l,
\end{equation*}
where $\C_l$ 
is the mean computational cost at level $l$ including the projection cost.
After rearranging \eqref{eq:PML} according to
$N_l$,  we obtain the following theorem.
 \begin{thm}
 \label{thm:lmlmc:sampling:cost}
 For given $\eps>0$,
the optimal $N_l$ and the computational cost 
by the PMLMC method
 are 
 \begin{equation}
 \label{eq:lmlmc:sampling:cos}
 \left\{
\begin{array}{l}
N_1 =
{{\eta^2}{\teps^{-2}}}
\paren{
\norm{ {u -   u_{1}} }+\norm{u}  }^{\f23}{\C_1}^{-\f23},
\\
\\
N_l =
{{\eta^2}{\teps^{-2}}}
\brace{2
\paren{
\norm{u-u_{l}}+
\norm{ {u -   u_{l-1}} }+\norm{u}  }}^{\f23}{\C_l}^{-\f23},
\quad l=2,\ldots,L,
\\
\\
\C\paren{\EPML{L}(u)} =
{{\eta^3}{\teps^{-2}}}.
\end{array}
\right.
\end{equation}
Here, the auxiliary variable 
$\teps=\eps - {\norm{u-u_1}}
-
2
\sum_{l=2}^{L} 
{\norm{u-u_l}}$
 should be positive
and another auxiliary variable $\eta$ is
\begin{equation*}
\eta=
\paren{
\norm{ {u -   u_{1}} }+\norm{u}  }^{\f23}{\C_1}^{\f13}
+
\sum_{l=2}^L 
\brace{2\paren{
\norm{u-u_{l}}+
\norm{ {u -   u_{l-1}} }+\norm{u}  }}^{\f23}{\C_l}^{\f13}.
\end{equation*}
\end{thm}
\begin{proof}
We can rearrange
\eqref{eq:PML} with respect to $N_l$ as follows
\begin{eqnarray*}
\norm{\D^\PML_L(u)}
&\le&
{\norm{u-u_1}}
+
2
\sum_{l=2}^{L} 
{\norm{u-u_l}}
+
N_1^{-\f12}\paren{\norm{u}+\norm{u-u_1}}
\\
&+&
2\sum_{l=2}^L N_l^{-\f12} \paren{\norm{u-u_{l}}+\norm{u-u_{l-1}}+\norm{u}}
\equiv g\paren{N_1,\ldots,N_L}.
\end{eqnarray*}
Set $f\paren{N_1,\ldots,N_L}=\sum_{l=1}^L N_l \C_l$.
The Lagrange's multiplier method 
\cite{arfken:1985} 
for $f$ under the constraint $g=\eps$,
gives us the result.
\end{proof}
From the condition on $\teps$, the natural inequalities $\norm{u-u_l}<\eps$ 
must hold for all $l=1,\ldots, L$.
We can give a concrete version of \thmref{thm:lmlmc:sampling:cost} 
by specifying the finite element space
as follows.
\begin{cor}
Let $V_l$ have only linear elements, i.e, 
 $\norm{u-u_l}=\bigO(h_l)$, or $\norm{u-u_l}\le C h_l$,
 and $h_l=2 h_{l+1}$ for $l=1,\ldots, L-1$,
 where $C$ is independent of $h_l$.
 Under the condition 
 $h_1 < {C^{-1}}{\eps}$, 
 we can make the following estimation
 \begin{equation*}
\C\paren{\E_{L}^\PML(u)}
\le 
4^{1+d}  \eps^{-2(1+d)} C^{2d} C'
 \brack{ \paren{\f\eps2+\norm{u}}^{\f23}
+
\sum_{l=2}^L \paren{\f{3\eps}{2^{l}} + \norm{u}}^{\f23} 2^{\f{2(1+ld)}{3}}}^3,
\end{equation*}
where 
the mean complexity at level $l$ is
$\C_l\le C' h_l^{-2d}$ 
for another constant $C'$ different from $C$
and the space dimension $d=1,2$.
Thus the complexity  increases in proportion to the power $2+2d$ of
the inverse of the desired error $\eps$.
 \end{cor}
 \begin{proof}
 Since $\norm{u-u_1}<\eps$ should hold,
the maximum mesh size $h_1$ at the coarsest grid would satisfy $h_1 < {C^{-1}}{\eps}$,
i.e., $h_L < 2^{1-L}{C^{-1}}{\eps}$. 
Choose $h_L=2^{-L}{C^{-1}}{\eps}$, or $h_1=2^{-1}{C^{-1}}{\eps}$, then
the resulting matrix to solve the problem is sparse. 
Inserting these values into
\eqref{eq:lmlmc:sampling:cos} completes the proof.
\end{proof}

\subsection{Completion of \lemref{lem:var:proj} using
Variance from Linear Interpolation}
\label{subsec:var}

Let $K$ be  a triangle in a triangulation $\T_{l-1}$ to form a piecewise
linear approximation space $V_{l-1}$. 
We make $\T_l$ by making 
4 sub-triangles $\Set{K_i}_{i=1}^4$ from
vertices $\Set{\vertex_i}_{i=1}^3$
and mid points $\Set{\midpoint_i}_{i=1}^3$ on edges of the triangle
$K$ as shown in \figref{fig:triangle}. Here, 
{\em mid points\/} are defined by
\begin{equation*}
\midpoint_k=\f{\vertex_i+\vertex_j}{2},\quad
k\ne i,\,i\ne j,\, j\ne k.
\end{equation*}
Any point $x$ in the triangle $K$ is expressed by
\begin{equation*}
x=\sum_{i=1}^3 \weight_i(x) \, \vertex_i,
\quad
\sum_{i=1}^3 \weight_i(x)=1,
\end{equation*}
where 
$\Set{\weight_i(x)}_{i=1}^3\subset\R$ are 
{\em weights\/} at $x$ in $K$,
called 
{\em barycentric coordinates\/},
determined by
proportional lengths as illustrated in \figref{fig:triangle}.
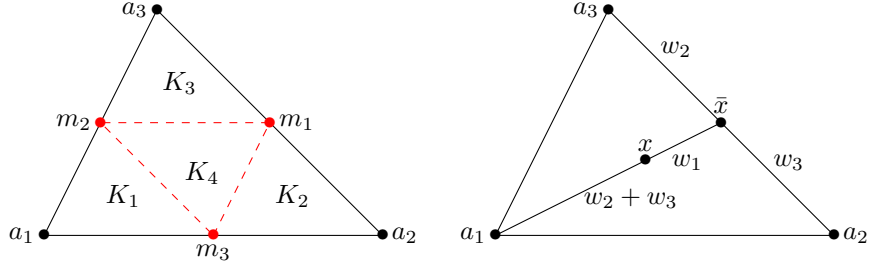
\begin{figure} [h!tbp]
\centering
	\begin{tikzpicture}[scale=1.5]
	\draw (0,0) -- (3,0) -- (1,2) -- (0,0);
	\foreach \Point in {(0,0), (3,0), (1,2)}{\node at \Point {\textbullet};}
	\foreach \Point in {(1.5,0), (2,1), (0.5,1)}{\node[red] at \Point {\textbullet};}
	\draw[red,dashed] (1.5,0) -- (2,1) -- (0.5,1) -- (1.5,0);
	\node[left] at (0,0) {$\vertex_1$};
	\node[right] at (3,0) {$\vertex_2$};
	\node[left] at (1,2) {$\vertex_3$};
	\node[below] at (1.5,0) {$\midpoint_3$};
	\node[right] at (2,1) {$\midpoint_1$};
	\node[left] at (0.5,1) {$\midpoint_2$};
	\node at (0.7,0.35) {$K_1$};
	\node at (2.2,0.35) {$K_2$};
	\node at (1.2,1.35) {$K_3$};
	\node at (1.4,0.55) {$K_4$};
	\draw (4,0) -- (7,0) -- (5,2) -- (4,0);
	\draw (4,0) -- (6,1);
	\foreach \Point in {(4,0), (7,0),(5,2),(6,1), (5.33,0.67)}{\node at \Point {\textbullet};}
	\node[left] at (4,0) {$\vertex_1$};
	\node[right] at (7,0) {$\vertex_2$};
	\node[left] at (5,2) {$\vertex_3$};
	\node[right,above] at (6,1) {$\bar x$};
	\node[above] at (5.33,0.67) {$x$}; 
	\node[below] at (5.2,0.5) {$w_2+w_3$}; 
	\node[below] at (5.7,0.8) {$w_1$}; 
	\node[right,above] at (5.6,1.5) {$w_2$};
	\node[right,above] at (6.6,0.5) {$w_3$};
	\end{tikzpicture}
\centering 
\caption{In {\em Left\/}, vertices $\Set{\vertex_i}_{i=1}^3$ and 
mid points $\Set{\midpoint_i}_{i=1}^3$ on edges of a triangle $K$ in $\T_{l-1}$
and
sub-triangles $\Set{K_i}_{i=1}^4\subset \T_l$ are shown in
black and red dots.
In {\em Right\/}, weights $\Set{\weight_i}_{i=1}^3$ are drawn
as proportional lengths
 for a point $x$ with an auxiliary point $\bar x$
on the opposite edge to the vertex $\vertex_1$.
}
\label{fig:triangle}
\end{figure}
For $u(\cdot;\om)\in V_l$, its linear projection $P u(\cdot;\om)\in V_{l-1}$ 
on a triangle $K\in\T_{l-1}$
is
\begin{equation}
\label{eq:proj:K}
Pu\big|_K(x;\om)=\sum_{i=1}^3 \weight_i(x) u(\vertex_i;\om),
\end{equation}
by the definition \eqref{eq:proj}.
Since $u$ is piecewise linear,
we can express $u$ on $K$ as
\begin{equation*}
u\big|_K(x;\om)=\sum_{j=1}^4 u\big|_{K_j}(x;\om)
=\sum_{j=1}^4 \sum_{i=1}^3 \weight_{ij}(x) u(\vertex_{ij};\om)
\end{equation*}
using $\Set{K_i}_{i=1}^4$ in \figref{fig:triangle}, 
with ordered points $\{\vertex_{ij}\}$ and weights $\{\weight_{ij}\}$
\begin{equation*}
\brack{\vertex_{ij}}=
\left[
\begin{array}{cccc}
  \vertex_1& \midpoint_3   & \midpoint_2 & \midpoint_1   \\
  \midpoint_3&   \vertex_2 & \midpoint_1  &\midpoint_2  \\
 \midpoint_2& \midpoint_1  & \vertex_3 & \midpoint_3   
\end{array}
\right],
\quad
\brack{\weight_{ij}}=
\left[
\begin{array}{cccc}
  \weight_{11}& \weight_{12}   & \weight_{13} & \weight_{14}   \\
  \weight_{21}& \weight_{22}   & \weight_{23} & \weight_{24}   \\
  \weight_{31}& \weight_{32}   & \weight_{33} & \weight_{34}   
\end{array}
\right].
\end{equation*}
Note that orientations of sub-triangles are determined by $\vertex_{ij}$ and
weights $\weight_{ij}$ have supports not on the whole $K$, but only on $K_j$,
which  are, for $i,j=1,2,3$,
\begin{eqnarray*}
\weight_{ij}(x) = \brack{2 \weight_i (x)- \delta_{ij}}\,\chi_{j} (x),
\quad
\weight_{i4}(x) = \brack{1 - 2 \weight_i (x)} \,\chi_{4}(x),
\end{eqnarray*}
where 
$\delta_{ij}$ is the Kronecker delta and
$\chi_j(x)$ is the characteristic function on $K_j$.
Using $\chi_j$, 
we express $Pu|_K$ as
\begin{equation*}
Pu\big|_K(x;\om)=\sum_{j=1}^4 \sum_{i=1}^3 
\weight_i(x) \chi_j(x) u(\vertex_i;\om).
\end{equation*}
The deviation $v$ of the difference 
$u-Pu$ and its mean $\E\brack{u-Pu}$
has the form
\begin{equation*}
{v}\big|_K(x;\om)
=
\paren{u-Pu-\E\brack{u-Pu}}\big|_K(x;\om)
=
\sum_{j=1}^4 v\big|_{K_j}(x;\om),
\end{equation*}
where $v|_{K_j}$ are
\begin{eqnarray*}
v\big|_{K_i}
&=&
2\paren{
 {\y_j} \weight_k
+
 {\y_k} \weight_j
},
\quad
i\ne j,\,j\ne k,\,k\ne i,
\\
v\big|_{K_4}
&=&
\sum_{i=1}^3
\brack{
\y_i
-
 {2\y_i}\weight_i
},
\end{eqnarray*}
for $i,j,k=1,2,3$
and
deviations 
$\y_i$ of
mid point values of $u(\cdot;\om)$
defined
by
\begin{equation*}
\y_i=u(\midpoint_i;\om)-Pu(\midpoint_i;\om)
-
\E\brack{u(\midpoint_i;\cdot)-Pu(\midpoint_i;\om)},\quad
i=1,2,3,
\end{equation*}
since $u(\vertex_i;\om)=Pu(\vertex_i;\om)$ 
by the property of $Pu$ from \eqref{eq:proj:K}.
We have 
\begin{equation*}
\int_K v^2\,dx 
=
\sum_{j=1}^4 \int_{K_j}v^2\,dx
\end{equation*}
where
\begin{eqnarray*}
\int_{K_i}v^2\,dx
&=&
\abs{K_i}\times
\f16
\paren{
\y_j^2
+
{\y_k}^2
+
{\y_j}{\y_k}
},
\quad
i\ne j,\,j\ne k,\,k\ne i,
\\
\int_{K_4}v^2\,dx
&=&
 \abs{K_4}
\times
\f16
\paren{
\sum_{i=1}^3 
{\y_i}^2
+
\sum_{i<j} \y_i\y_j 
},
\end{eqnarray*}
for $i,j,k=1,2,3$.
Here, $\abs{K_i}$ denotes the area of $K_i$ and has the same value, 
a quarter of $\abs{K}$ by
mid point rule. 
For $H^1(D)$ norm, we calculate $\Grad v$ on $K_i$ as
\begin{eqnarray*}
\Grad v\big|_{K_i}
&=&
2\paren{\y_j \Grad\weight_k + \y_k \Grad\weight_j},
\quad
i\ne j,\,j\ne k,\,k\ne i,
\\
\Grad v\big|_{K_4}
&=&
-2\sum_{i=1}^3
{\y_i} \Grad\weight_i,
\end{eqnarray*}
for $i,j,k=1,2,3$.
Their integrals are
computed 
by
\begin{eqnarray*}
\int_{K_i} \abs{\Grad\weight_j}^2\,dx
&=&\f{\abs{K_i}}{4\abs{K}^2} \, \e_j^2
\\
\int_{K_i} \Grad\weight_j\cdot\Grad\weight_k\,dx
&=&
-\f{\abs{K_i}}{4\abs{K}^2}  \,\e_j \e_k \cos\theta_l,
\quad
j\ne k,\,k\ne l,\, l\ne j,
\end{eqnarray*}
where
$\e_i$ is the length of the edge opposite to the vertex $\vertex_i$ and
 $\theta_i$ is the angle at $\vertex_i$ in $K$
 for $j,k,l=1,2,3$.
After arranging the summation, we obtain the following
\begin{eqnarray*}
\int_{K}v^2\,dx
&\le &
\f{5}{24}\abs{K}\sum_{i=1}^3
{\y_i^2}
\le
\regconst \, \diam^2
\sum_{i=1}^3
{\y_i^2},
\\
\int_{K}\abs{\Grad v}^2\,dx
&\le &
\f{1}{\abs{K}}
\sum_{i=1}^3 \e_i^2 
\,
\sum_{i=1}^3
{\y_i^2}
\le
\regconst
\sum_{i=1}^3
{\y_i^2}.
\end{eqnarray*}
Here,
$\regconst \ge 5/{48}$ is a constant for the 
{\em regular triangulation\/} such that
 the diameter $\indiam$ of the incircle 
 and 
 the longest side $\diam$ of $K$
satisfy a relation 
${\diam}/{\indiam}\le \regconst/6$ for all $K\in\T_{l-1}$,
see
Zl\'amal's minimal angle condition \cite{zlamal:1968}, 
Ciarlet's
inscribed ball condition \cite{ciarlet:1978},
and
the maximum angle condition by
Babu\v{s}ka et al.\
\cite{BA:1976} and
Jamet \cite{jamet:1976}.
Now, the expectation of the square of $H^1(D)$-norm of the deviation is
\begin{eqnarray*}
\E\brack{\norm{v}_V^2}
&=&
\E\brack{
\norm{v}_{L^2(D)}^2
+
\norm{\Grad v}_{L^2(D)^2}^2}
\\
&\le&
\regconst
\paren{ \diam^2 +1 } \sum_{K\in\T_{l-1}}\sum_{i=1}^3
\E\brack{\y_i^2}
\\
&=&
\regconst
\paren{ \diam^2 +1 } \sum_{K\in\T_{l-1}}\sum_{i=1}^3
\Var\brack{\paren{u-Pu}(\midpoint_i;\cdot)}.
\end{eqnarray*}
Note that the left-hand side is the variance of $u-Pu$ and
 variances at the right-hand side are  usual variances for discrete values.
The proof of \lemref{lem:var:proj} is done
by replacing $\diam$, $u-Pu$ and $\midpoint_i$ with 
$\diam_{l-1}$, $u_l - P_l u_l$ and $\midpoint_i(K)$, respectively.

\section{Numerical Simulations}
\label{sec:NS}

We use the finite element method in the piecewise linear function
space to approximate the Darcy flow \eqref{eq:darcy}
with a fixed coefficient.
We examine the performance of  MLMC and PMLMC methods in computing
the mean values of pressure fields of the Darcy flow
\eqref{eq:darcy}
 in $D=(0,1)^d$ for $d=1,2$.
 As \cite{CGST:2011}, we use the covariance operator 
 \eqref{eq:cov} to make the logarithm of the coefficient
 $k$ in \eqref{eq:darcy} by the KL expansion
 for the case $\lambda=0.1$ and
$\sigma^2=1$.
 
\subsection{Results in 1D}
\label{subsec:1D}

For the KL expansion, we use $\dx=1/8192$ to generate eigenvalues
and eigenfunctions up to $1000$ modes. 
We use the mesh size $h=1/4096$ for the fine grid 
  and $h=1/2048$ for the coarse grid. 
  The  mean of
$200000$ samples by the MC method
 is regarded as $\E\brack u$ for the comparison
of 
MLMC and PMLMC methods.

We measure the actual CPU time in seconds 
on a $3.07$ GHz Intel Core i7 processor with 12 GB of RAM
using a \Matlab\ 
\cite{MATLAB:2012} code.
It takes about 2 days and 14 hours in CPU time for the calculation
of $200000$ sample solutions.
We observe that the PMLMC method speeds up 
 the correction step at the fine grid
by at least
$12\%$ compared to the MLMC method
as tabulated in {\em Left} of \tabref{tab:cputime:h1error:1D}
and depicted in \figref{fig:cputime:1D}.
The PMLMC method
keeps the almost same order of errors
as shown 
in {\em Right} of \tabref{tab:cputime:h1error:1D}.
We use a fixed number of samples at the fine grid
for MLMC and PMLMC methods as described in
the caption of {\em Right} of \tabref{tab:cputime:h1error:1D}
and illustrate the results in \figref{fig:h1error:1D}.

\begin{table}[h]
\caption{In {\em Left\/},
CPU time (sec) versus $N_2$ (number of samples at the fine grid) 
for MLMC and PMLMC methods in 1D.
In {\em Right\/},
$H^1(D)$ error versus $N$ (total number of samples) 
for MC, MLMC and PMLMC methods in 1D. 
The number of samples at the fine grid is fixed as $N_2=50$
for MLMC and PMLMC methods.}
\label{tab:cputime:h1error:1D}
{\small
\centering
\begin{tabular}{|c|c|c|}
\hline
\multicolumn{3}{|c|}{CPU Time (sec)}
\\
\hline
$N_2$ & MLMC & PMLMC
\\
\hline
50 & 57 &49
    \\
\hline
200 & 227 & 198\\
\hline
 800 & 904 & 780\\
\hline
 3200 & 3626 & 3154\\
\hline
 12800 & 14495 & 12574
\\
\hline
\end{tabular}
\begin{tabular}{|c|c|c|c|}
\hline
\multicolumn{4}{|c|}{$H^1(D)$ error}
\\
\hline
$N$ & MC &MLMC & PMLMC
\\
\hline
100 & $9.0261\times10^{-2}$ &  $1.2547\times10^{-1}$ & $1.2548\times10^{-1}$   \\
\hline
250 & $7.5006\times10^{-2}$ & $7.6602\times10^{-2}$ & $7.6599\times10^{-2}$\\
\hline
 850 & $3.3139\times10^{-2}$ & $3.2706\times10^{-2}$ & $3.2708\times10^{-2}$\\
\hline
 3250 & $2.1194\times10^{-2}$ & $2.2330\times10^{-2}$ & $2.2328\times10^{-2}$\\
\hline
 12850 & $8.5928\times10^{-3}$ & 
$9.3939\times10^{-3}$ & 
$9.3952\times10^{-3}$
\\
\hline
\end{tabular}
}
\end{table}


 \begin{figure}[h]
\centering
{
\includegraphics[totalheight=0.4\textheight]{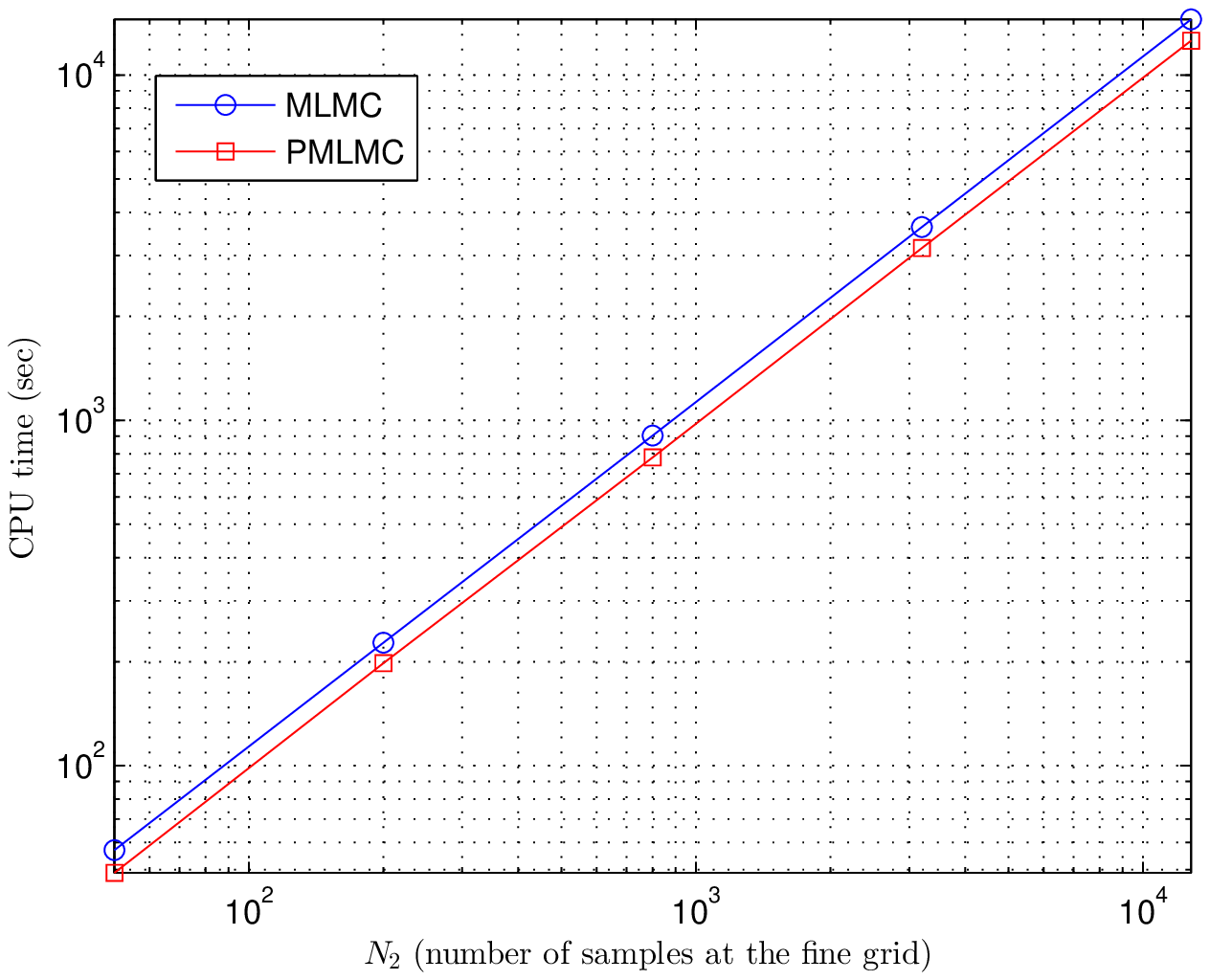}
}
\caption{
CPU time (sec) versus $N_2$ (number of samples at the fine grid) 
for MLMC and PMLMC methods in 1D 
corresponding to {\em Left\/} of \tabref{tab:cputime:h1error:1D}.
}
 \label{fig:cputime:1D}
\end{figure}

 \begin{figure}[h]
\centering
{
\includegraphics[totalheight=0.4\textheight]{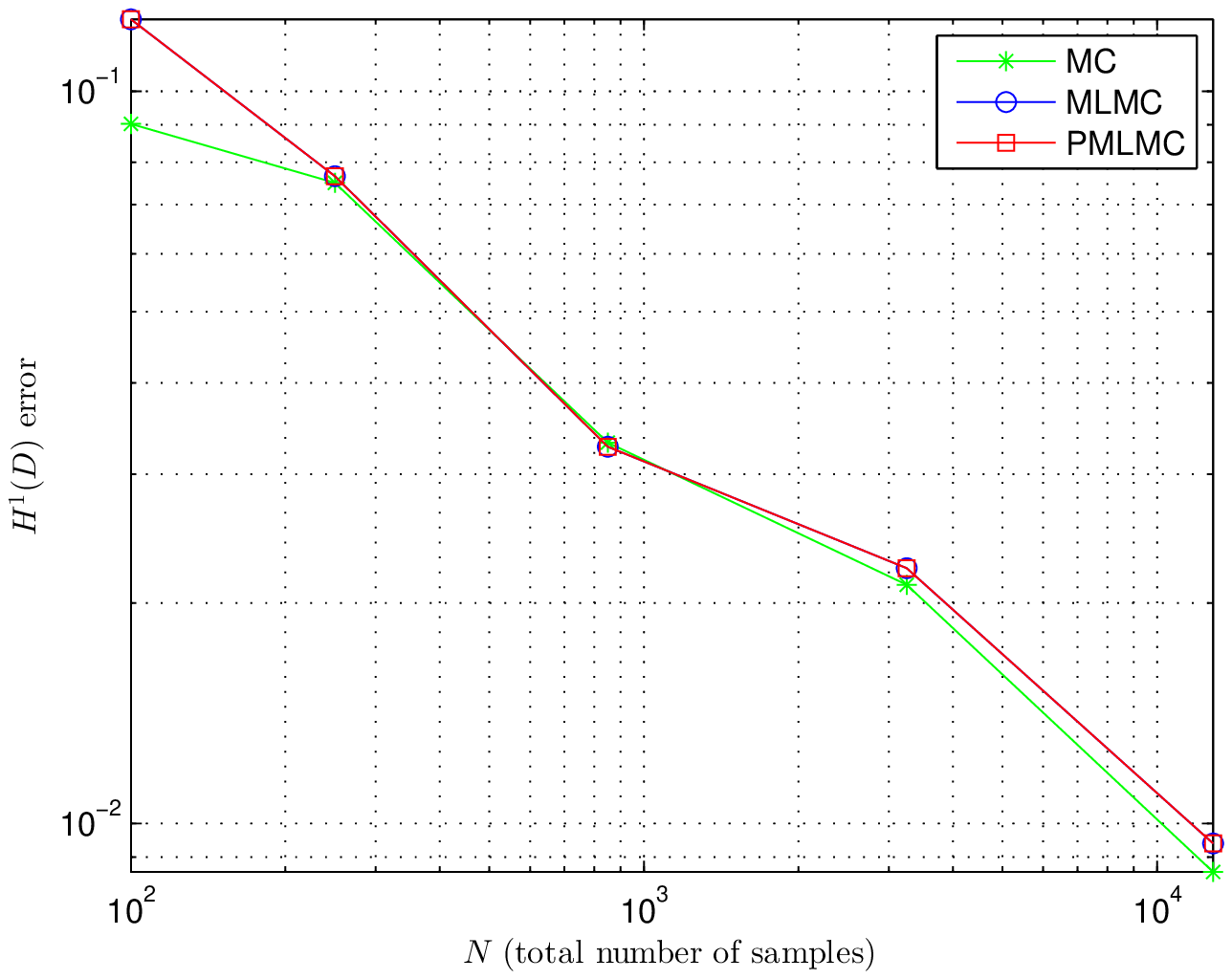}
}
\caption{
$H^1(D)$ error versus $N$ (total number of samples) 
for MC, MLMC and PMLMC methods in 1D 
corresponding to {\em Right\/} of \tabref{tab:cputime:h1error:1D}. 
The number of samples at the fine grid is fixed as $N_2=50$
for MLMC and PMLMC methods.
}
 \label{fig:h1error:1D}
\end{figure}

\subsection{Results in 2D}
\label{subsec:2D}

We make eigenvalues in 2D from 
the tensor products of those in 1D, 
see Cliffe et al.\
\cite{CGST:2011} for details.
We re-order eigenvalues in magnitude and truncate them
up to $1000$ modes in the rectangular grid of size $1/128$. 

We use the triangular mesh to compute approximations.
The coarse grid is
generated by \DistMesh\
\cite{ps:2004} with a mesh size 
$h=0.0442$.
The fine mesh is made from the coarse grid using the mid point rule,
and has a mesh size
$h=0.0221$.
The fine and coarse mesh grids form a hierarchical mesh system and make it possible to
reduce the error by decreasing the unmatched node points between 
coarse and fine grids.
Numbers of elements and nodes  are
$2934$, $1262$ at the coarse grid
and
$9576$, $4917$ 
at the fine grid, which means the fine grid has  more than triple elements and nodes
compared to those in the coarse grid.

We use 
the same machine 
with a \Matlab\ code
in 2D.
It takes 
6 days and 15 hours 
in CPU time for the calculation
of $30000$ sample solutions.
%
The PMLMC method speeds up 
  at the fine grid
by at least
$18\%$
%
as tabulated in {\em Left} of \tabref{tab:cputime:h1error:2D}
and depicted in \figref{fig:cputime:2D}.
In {\em Right} of \tabref{tab:cputime:h1error:2D},
we doubles numbers of samples at  the fine grid
while quadruples those at the coarse grid. We illustrate the results
in \figref{fig:h1error:2D}.
The main reason to increase $N_2$ is due to the
slow convergence rate. We expect that we
can fix $N_2$ when the mesh size is small enough
as in the case of 1D.

\begin{table}[h]
\caption{In {\em Left\/},
CPU time (sec) versus $N_2$ (number of samples at the fine grid) 
for MLMC and PMLMC methods in 2D.
In {\em Right\/},
$H^1(D)$ error versus $N$ (total number of samples) 
for MC, MLMC and PMLMC methods in 2D. 
$N_2$ are $50, 100, 200, 400, 800$
for MLMC and PMLMC methods.}
\label{tab:cputime:h1error:2D}
{\small
\centering
\begin{tabular}{|c|c|c|}
\hline
\multicolumn{3}{|c|}{CPU Time (sec)}
\\
\hline
$N_2$ & MLMC & PMLMC
\\
\hline
50 & 1179 & 957
    \\
\hline
200 & 4728 & 3838
\\
\hline
 800 & 18911 & 15351\\
\hline
 3200 & 75644 & 61405
 \\
\hline
 12800 & 302577 & 245615
\\
\hline
\end{tabular}
\begin{tabular}{|c|c|c|c|}
\hline
\multicolumn{4}{|c|}{$H^1(D)$ error}
\\
\hline
$N$ & MC &MLMC & PMLMC
\\
\hline
100 &$6.7379\times10^{-2}$ &$9.1351\times10^{-2}$ &  $9.1353\times10^{-2}$   \\
\hline
300 &$5.5589\times10^{-2}$ &$6.2854\times10^{-2}$ & $6.2687\times10^{-2}$
\\
\hline
 1000 & $2.4383\times10^{-2}$&$3.7902\times10^{-2}$ & $3.7713\times10^{-2}$\\
\hline
 3600 &$1.0850\times10^{-2}$ &$3.0330\times10^{-2}$ & $3.0255\times10^{-2}$
 \\
\hline
 13600 & $5.2386\times10^{-3}$ & 
$2.9398\times10^{-2}$ & 
$2.9226\times10^{-2}$
\\
\hline
\end{tabular}
}
\end{table}


 \begin{figure}[h]
\centering
{
\includegraphics[totalheight=0.4\textheight]{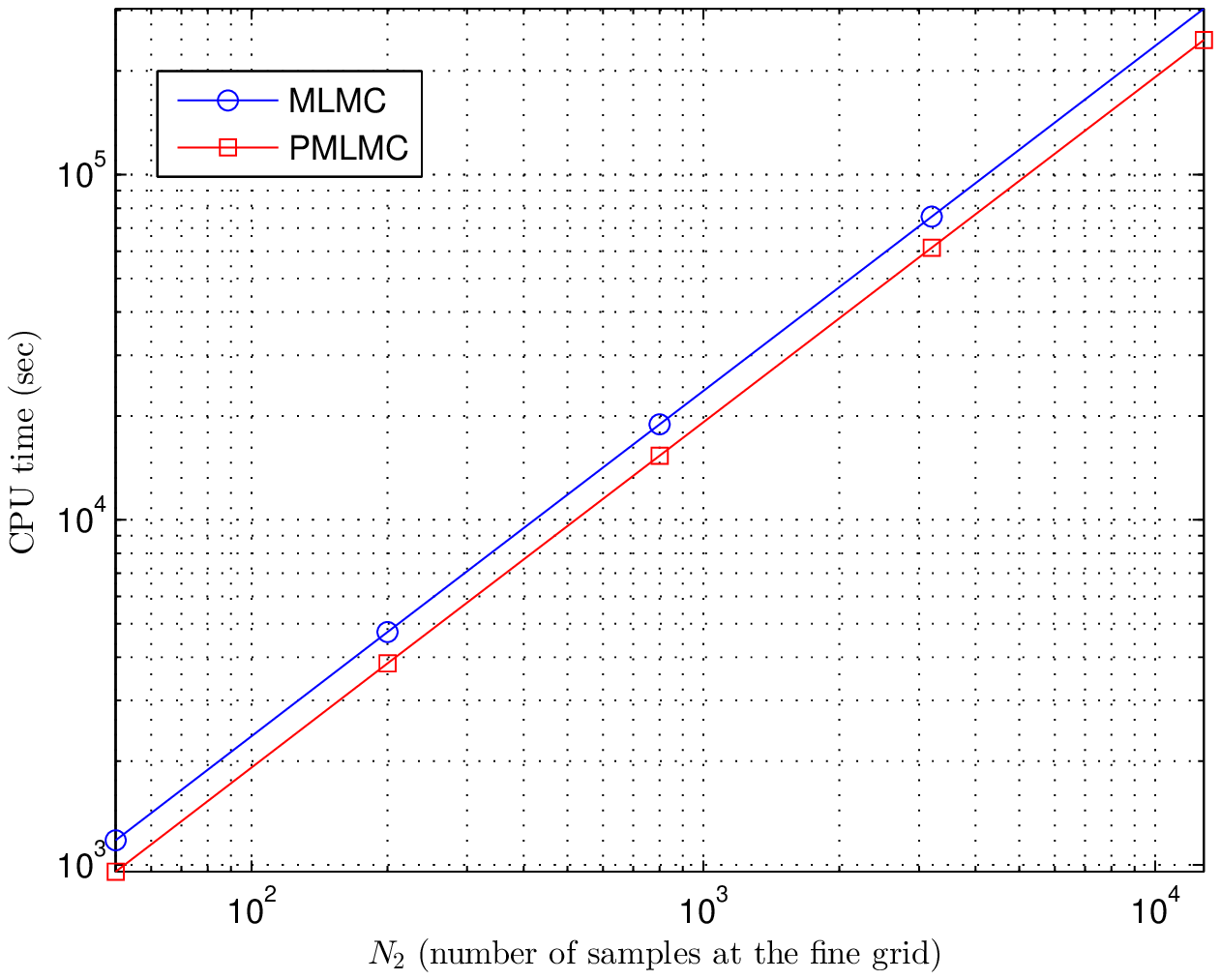}
}
\caption{
CPU time (sec) versus $N_2$ (number of samples at the fine grid) 
for MLMC and PMLMC methods in 2D 
corresponding to {\em Left\/} of \tabref{tab:cputime:h1error:2D}.
}
 \label{fig:cputime:2D}
\end{figure}
 \begin{figure}[h]
\centering
{
\includegraphics[totalheight=0.4\textheight]{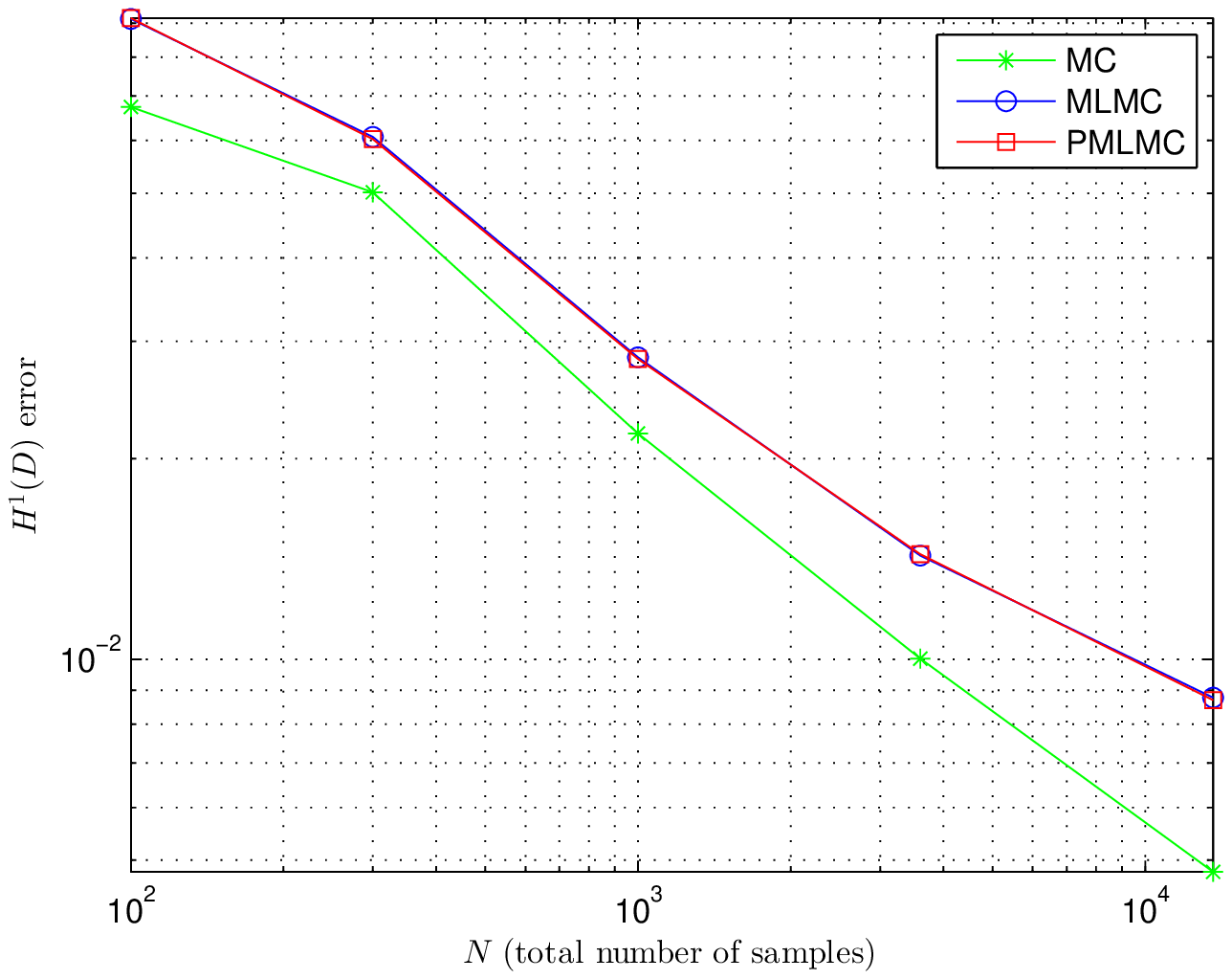}
}
\caption{
$H^1(D)$ error versus $N$ (total number of samples) 
for MC, MLMC and PMLMC methods in 2D 
corresponding to {\em Right\/} of \tabref{tab:cputime:h1error:2D}. 
$N_2$ are $50, 100, 200, 400, 800$
for MLMC and PMLMC methods. 
}
 \label{fig:h1error:2D}
\end{figure}

\subsection{Cost Savings}
 For comparison of CPU time, we tabulate CPU times
for MC, MLMC and PMLMC methods in
\tabref{tab:cputime}.
We illustrate the results in 
\figref{fig:cputime:comp:1D}
and
\figref{fig:cputime:comp:2D}
for 1D and 2D, respectively.
 In 1D,
the
computational cost savings by the PMLMC method are constant,
since
the projection 
occurs only $50$ times.

In the fine grid, the PMLMC method spends almost same time as the MC method does, while
the MLMC method does 14 \%\  
in 1D and 20 \%\
in 2D
more time than the MC method does,
as tabulated
in {\em Left\/} of 
\tabref{tab:cputime:h1error:1D}
 and 
 \tabref{tab:cputime:h1error:2D}
for 1D
and
2D, respectively.
We can say that the PMLMC method saves the computational cost
further than the MLMC method.
This is due to the solution procedure for
the MLMC method in the coarse grid. To the contrary, the PMLMC method 
uses the projection, a cheaper procedure compared to the MLMC method
in the view of cost savings.

\begin{table}[h]
\caption{
CPU time (sec) versus $N$ (total number of samples) 
for MC, MLMC and PMLMC methods in 1D ({\em Left}) and 2D ({\em Right}). 
The number of samples at the fine grid is fixed as $N_2=50$
for MLMC and PMLMC methods.}
\label{tab:cputime}
\centering
\begin{tabular}{|c|c|c|c|}
\hline
\multicolumn{4}{|c|}{CPU Time (sec) in 1D}
\\
\hline
$N$ & MC &MLMC & PMLMC
\\
\hline
100 &$99$ &$68$ &  $60$   \\
\hline
250 &$246$ &$100$ & $93$
\\
\hline
 850 & $837$&$227$ & $218$\\
\hline
 3250 &$3201$ &$734$ & $726$
 \\
\hline
 12850 &$12729$ &$3103$ & $3091$
\\
\hline
\end{tabular}
\begin{tabular}{|c|c|c|c|}
\hline
\multicolumn{4}{|c|}{CPU Time (sec) in 2D}
\\
\hline
$N$ & MC &MLMC & PMLMC
\\
\hline
100 &$1918$ &$1409$ &  $1188$   \\
\hline
300 &$5754$ &$3263$ & $2825$
\\
\hline
 1000 & $19187$&$8425$ & $7508$\\
\hline
 3600 &$68995$ &$24100$ & $22243$
\\
\hline
13600& $262352$& $77108$ & $73220$
\\
\hline
\end{tabular}
\end{table}


 \begin{figure}[h]
\centering
{
\includegraphics[totalheight=0.4\textheight]{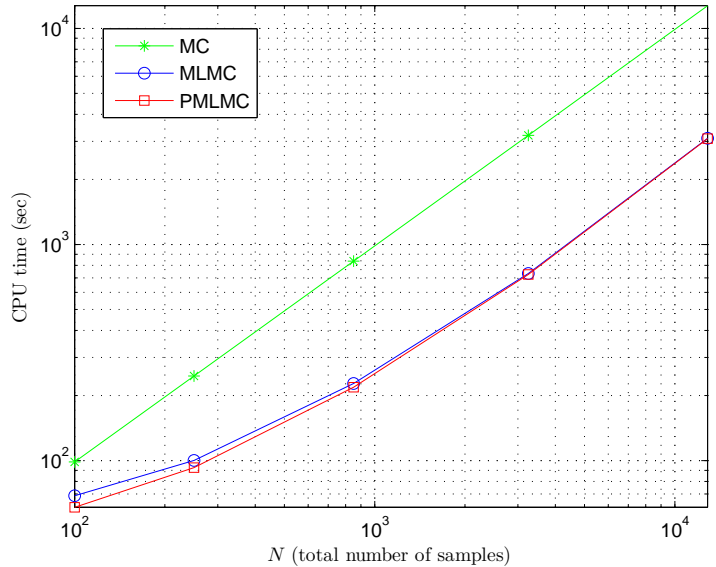}
}
\caption{
CPU time (sec) versus $N$ (total number of samples) 
for MC, MLMC and PMLMC methods in 1D
corresponding to {\em Left\/} of \tabref{tab:cputime}. 
The number of samples at the fine grid is fixed as $N_2=50$.
}
 \label{fig:cputime:comp:1D}
\end{figure}

 \begin{figure}[h]
\centering
{
\includegraphics[totalheight=0.4\textheight]{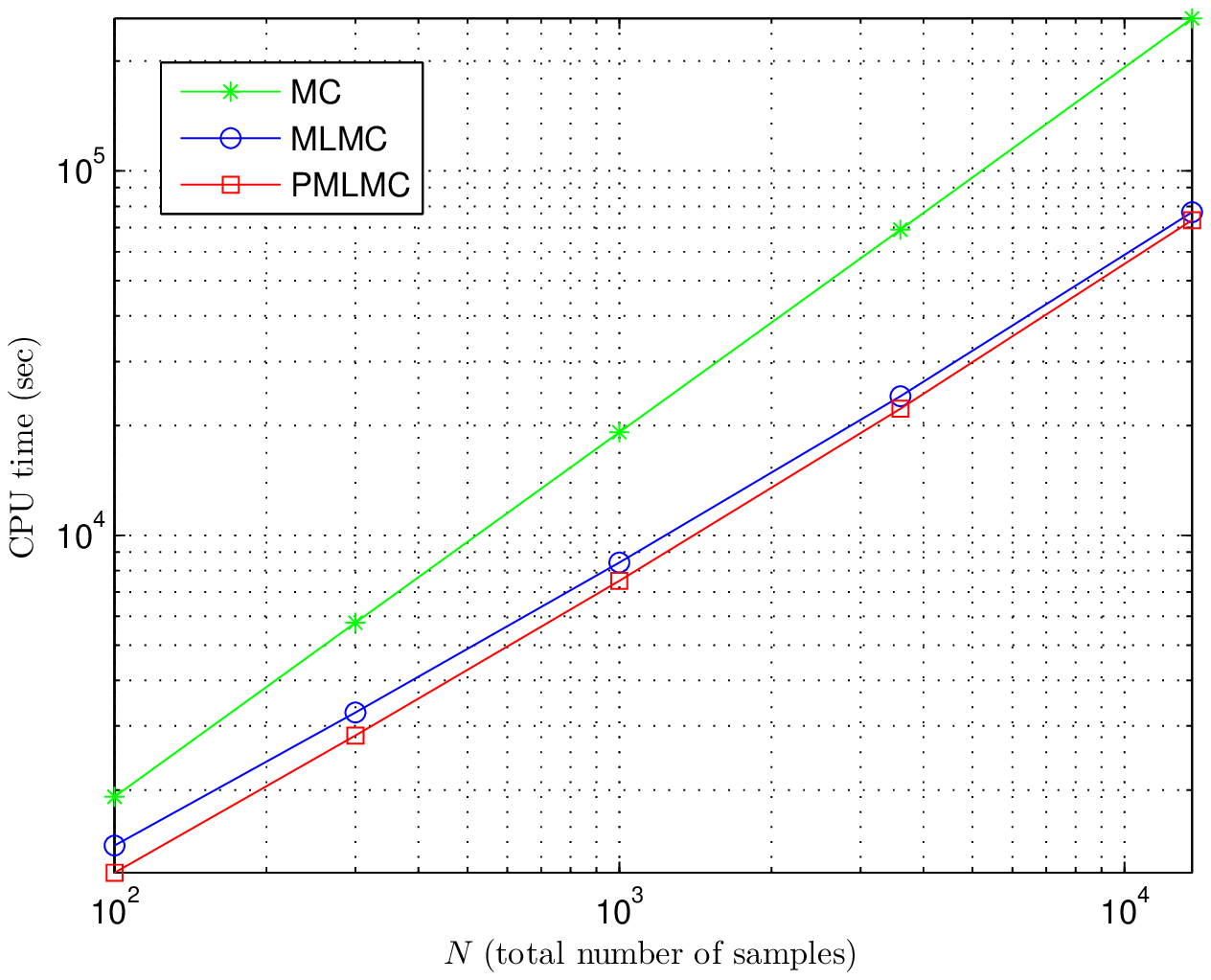}
}
\caption{
CPU time (sec) versus $N$ (total number of samples) 
for MC, MLMC and PMLMC methods in 2D
corresponding to {\em Right\/} of \tabref{tab:cputime}.
$N_2$ are $50, 100, 200, 400, 800$
for MLMC and PMLMC methods. 
}
 \label{fig:cputime:comp:2D}
\end{figure}

\section{Conclusion}
\label{sec:con}

From the error estimations in \secref{sec:order},
we confirm that
the order of convergence of the MC method is $1/2$ which means that
we need quadruple samples to decrease the error in half in
the numerical simulation. 
 MLMC and PMLMC methods are faster than the
MC method and show the same order of convergence $1/2$.
When the most computation occurs at the coarse grid,
the computational cost due to the increase of samples 
does not increase, rather decreases proportional to the ratio
of the number of samples at the coarse grid to that at the fine grid.
In the PMLMC method, we use the projection from the fine grid to the coarse grid
to replace the approximation at the coarse grid in the MLMC method,
which results in 
the reduction of
the computational cost further than the MLMC method.

The PMLMC method upgrades values at mid points of edges in the coarse grid
when we use a hierarchical grid structure under mid point refinement scheme.
This means that we can use small number of samples during the correction procedure
in the fine grid,
since the main structure of the mean value is estimated at the coarse grid
by illustration of convergence of error in
 \figref{fig:h1error:1D}
and \figref{fig:h1error:2D}.
Non-conforming finite element methods can be used for the PMLMC method
since the PMLMC method does not require the inclusion of two consecutive
approximate spaces.

The variance analysis and optimal number of samples with 
illustrations of them 
through  numerical simulations in 1D and 2D are shown for the completeness of
the paper.
In the near future,
we would present the results on 
MLMC and PMLMC methods combined with conforming and non-conforming
finite elements applied to
the Helmholtz 
equation with a random coefficient and wave equation with a white noise.

\bibliographystyle{spmpsci}      
\bibliography{pmlmc}   

\end{document}